\documentclass{article}
\usepackage[a4paper, total={7in, 10in}]{geometry}
\usepackage{amsmath,amsthm,amssymb,mathrsfs,amstext, titlesec,enumitem, comment, graphicx, color, xcolor, stmaryrd,mathabx}
\usepackage{hyperref}
\usepackage{xpatch}
\usepackage{tikz-cd}
\hypersetup{colorlinks,linkcolor={magenta},citecolor={blue}}
\usepackage{xpatch}
\usepackage{tikz-cd}
\makeatletter
\def\Ddots{\mathinner{\mkern1mu\raise\p@
\vbox{\kern7\p@\hbox{.}}\mkern2mu
\raise4\p@\hbox{.}\mkern2mu\raise7\p@\hbox{.}\mkern1mu}}
\makeatother

\titleformat{\section}
{\normalfont\Large\bfseries}{\thesection}{1em}{}
\titleformat*{\subsection}{\Large\bfseries}
\titleformat*{\subsubsection}{\large\bfseries}
\titleformat*{\paragraph}{\large\bfseries}
\titleformat*{\subparagraph}{\large\bfseries}

\makeatother
\theoremstyle{Theorem}
\newtheorem{thm}{Theorem}[section]

\newtheorem{con}[thm]{Conjecture}

\theoremstyle{definition}

\newcommand{\N}{\mathbb{N}}
\newcommand{\Z}{\mathbb{Z}}

\newcommand{\F}{\mathcal{F}}

\date{\vspace{-5ex}}

\begin{document}

\title{{\bf Restricted van der Waerden theorem for nilprogressions}}
\author{ 
Sayan Goswami\\  \textit{ sayan92m@gmail.com}\footnote{Ramakrishna Mission Vivekananda Educational and Research Institute, Belur Math,
Howrah, West Benagal-711202, India.}
}


\makeatother

\providecommand{\corollaryname}{Corollary}
\providecommand{\definitionname}{Definition}
\providecommand{\examplename}{Example}
\providecommand{\factname}{Fact}
\providecommand{\lemmaname}{Lemma}
\providecommand{\propositionname}{Proposition}
\providecommand{\remarkname}{Remark}
\providecommand{\theoremname}{Theorem}

\maketitle
\begin{abstract}
In $[Adv.\, Math.,\, 321\, (2017)\, 269–286]$, using the theory of ultrafilters, J. H. Johnson Jr., and F. K. Richter proved the nilpotent polynomial Hales-Jewett theorem. Using this result they proved the restricted version of the van der Waerden theorem for nilprogressions of rank $2$ and conjectured that this result must hold for arbitrary rank. In this article, we give an affirmative answer to their conjecture.
 

\end{abstract}

\section{Introduction}
Arithmetic Ramsey theory deals with the monochromatic patterns found in any given finite coloring of the
integers or of the natural numbers $\N$. Here, ``coloring” means disjoint partition, and a set is called ``monochromatic” if it is included in one piece of the partition. Let $\F$ be a family of finite subsets of $\N.$ If for every finite coloring of $\N$, there exists a monochromatic member of $\F$, then such a family $\F$ is called a partition regular family. So basically, Ramsey theory is the study of the classifications of partitioned regular families. Arguably one of the first substantial developments in this area of research was due to Van der Waerden in $1927$ when he proved the following theorem.
\begin{thm}[\textbf{Van der Waerden theorem}, \cite{vdw}]\label{ Van der Waerden}
For any finite coloring of the natural numbers one always finds arbitrarily long monochromatic arithmetic progressions. In other words, the set of all arithmetic progressions of finite length is partitioned regularly.
\end{thm}

In \cite{26}, Spencer (independently in \cite{23}, by Ne\v{s}et\v{r}il, and V. R\"{o}dl) proved a restricted version of the Van der Waerden theorem:  there exists a set $V\subset \N$ containing no $k+1$ term arithmetic progressions and such that for any partition of $V$ into finitely many classes, some class must contain a $k$-term arithmetic progression. To address the conjecture for nilprogressions we need to recall some notions.  For $k\geq 2,$ define  $\sum_{<k}$  to be the collection of all words $w(*_1,*_2,\cdots ,*_d)$ in the letters $*_1,*_2,\cdots ,*_d$ such that every letter $*_i$ appears at most $(k-1)$ times. For any given word $w(*_1,*_2,\cdots ,*_d)$ and elements $x_1,x_2,\cdots ,x_d$ in a group $(G,\cdot ),$ let $w(x_1,x_2,\cdots ,x_d)$ 
 be  the group element of $G$ obtained by replacing all occurrences of the variable $*_i$ in the word $w(*_1,*_2,\cdots ,*_d)$ with $x_i.$  A \textbf{nilprogression of step $s$, length $k$ and rank $d$} is a set of the form  
 $A=\left\lbrace w(x_1,x_2,\cdots ,x_d)a:w\in \sum_{<k+1}\right\rbrace,$
 where $a,x_1,x_2,\cdots ,x_d$  are elements in an $s$-step nilpotent group $G$. If $|A|=|\sum_{<k+1}|$, then $A$ is called \textbf{non-degenerated nilprogression.} In \cite{adv} Johnson Jr. and Richter posed the following conjecture.

 \begin{con}[\textbf{Restricted van der Waerden theorem for nilprogressions}]\label{c}
 For every $k\geq 1,$ and $d\geq 2$ there exists a $k$-step nilpotent group $(G,\cdot )$ with $d$ generators and a set $V\subset G$ with the property that $V$ does not contain any non-degenerated nilprogressions of step $k$, length $k+1$ and rank $d$ but for any partition of $V$ into finitely many classes, some class contains a non-degenerated nilprogressions of step $k$, length k and rank $d$.
\end{con}

In \cite[Theorem B]{adv}, they proved this conjecture for $d=2$ using the Nilpotent polynomial Hales-Jewett theorem (\cite[Theorem A]{adv}), which was a conjecture of  Bergelson, and Leibman \cite{bl3}. In this article, we construct a $k$ step nilpotent group with $d$ generators, where we can explicitly calculate the evaluations of all words and then proceeding similarly to their proofs one can obtain the desired result. To avoid unnecessary technical difficulties we omit the entire literature on the Polynomial Hales-Jewett theorem (see \cite{bl2, adv}) and the algebra of Stone-\v{C}ech compactification (see \cite{hs}). That's why we also omit the full proof, but rather we limit ourselves to construct our required groups, and then one may adapt the technique of the rest of the proof as in the proof of \cite[Theorem B]{adv}.

\section{Our proof}
For a long time, the cartesian product trick has been proven to be a very useful technique to estimate bounds in additive combinatorics (see \cite{tao}). In our proof, we adapt this trick to construct our desired group.

\begin{proof}[Proof of Conjecture \ref{c}:] We divide the proof into two cases where $d$ is even or odd.\\
\vspace{.1 cm}

   \noindent \underline{\textbf{$d$ is even:}}  Let $\Z[x]$ denote the collection of all polynomials with integer coefficients. For each $i\in \{1,\cdots ,d/2\},$ define $R_i,S_i:\Z[x]^{d/2}\rightarrow \Z[x]^{d/2}$ as

\begin{enumerate}
    \item $R_i(P_1(x),\cdots ,P_{d/2}(x))=(P_1(x),\cdots ,P_{i-1}(x),P_i(x)+x^k,P_{i+1}(x),\cdots ,P_{d/2}(x))$
    \item $S_i(P_1(x),\cdots ,P_{d/2}(x))=(P_1(x),\cdots ,P_{i-1}(x),P_i(x+1),P_{i+1}(x),\cdots ,P_{d/2}(x))$
\end{enumerate}
   Let $G$ be the group generated by $\{R_i,S_i:1\leq i \leq d/2\}.$ For each $i\in \{1,2,\cdots ,d/2\},$ let $G_i$ be the group generated by $R_i$ and $S_i.$ Let $R,S:\Z[x]\rightarrow \Z[x]$ be the map defined by $R(p(x))=p(x)+x^k$ and $S(p(x))=p(x+1).$ As $R_i$ and $S_i$ are two maps which are nothing but $R$ and $S$ resp. on the $i^{th}$ coordinate and are identity over the rest of the coordinates, we have $G_i$ is isomorphic to the group generated by $R$ and $S.$ But the group generated by $R$ and $S$ is a $k$-step nilpotent group\footnote{this is the exactly same group that was considered in the proof of \cite[Theorem B]{adv}.}, and so our $G_i's$ are.
   Note if $i\neq j,$ $R_i,S_i$ commutes with $R_j,S_j,$ and as each $G_i$ is a $k$-step nilpotent group,
   $G$ is a $k$-step nilpotent group.

   Let $\sum_{<k+1}$ denote the collection of all words in the letters $*_1,\cdots ,*_d,$ where each $*_i$ appears at most $k$ times. Our claim is that for all $w_1,w_2\in \sum_{<k+1},$ if $w_1(R_1,S_1,\cdots ,R_{d/2},S_{d/2})=w_2(R_1,S_1,\cdots ,R_{d/2},S_{d/2})$ then $w_1(*_1,\cdots ,*_d)=w_2(*_1,\cdots ,*_d).$
   Note that for any $w\in \sum_{<k+1},$ one can write $w(R_1,S_1,\cdots ,R_{d/2},S_{d/2})=w^1(R_1,S_1)\cdots w^{d/2}(R_{d/2},S_{d/2})$ where each $w^i$ is a a word in $G_i$ with length $k$, rank $2.$

   Now $w_1(R_1,S_1,\cdots ,R_{d/2},S_{d/2})(x^k,\cdots ,x^k)=w_2(R_1,S_1,\cdots ,R_{d/2},S_{d/2})(x^k,\cdots ,x^k)$ implies $$(w_1^1(R_1,S_1)x^k,\cdots ,w_1^{d/2}(R_{d/2},S_{d/2})x^k)=(w_1^1(R_1,S_1)x^k,\cdots ,w_2^{d/2}(R_{d/2},S_{d/2})x^k).$$

   But then in light of the proof of \cite[Theorem B]{adv}, we have $w_1=w_2.$ This shows that $G$ admits non-degenerated nilprogressions of length $k$ and rank $d.$ As each rank $d$ word in $G$ is a product of rank $2$ words of $G_i's$ (at least one of them $G_1$), and as in \cite[Theorem B]{adv}, it has been proven that they do not admit non-degenerated nilprogressions of length $(k+1)$ and rank $2$, our group does not contain nilprogressions of length $(k+1)$ and rank $2$.

   Now the rest of the proof is similar to the proof of \cite[Theorem B]{adv}, so we leave it to the reader.

   \vspace{.2 cm}

   \noindent \underline{\textbf{$d$ is odd:}}  Whenever $d$ is odd, we need a little modification. For each $i\in \{1,\cdots ,(d-1)/2\},$ define $R_i,S_i,R_d:\Z[x]^{(d+1)/2}\rightarrow \Z[x]^{(d+1)/2}$ as

\begin{enumerate}
    \item $R_i(P_1(x),\cdots ,P_{(d+1)/2}(x))=(P_1(x),\cdots ,P_{i-1}(x),P_i(x)+x^k,P_{i+1}(x),\cdots ,P_{(d+1)/2}(x))$
    \item $S_i(P_1(x),\cdots ,P_{(d+1)/2}(x))=(P_1(x),\cdots ,P_{i-1}(x),P_i(x+1),P_{i+1}(x),\cdots ,P_{(d+1)/2}(x))$
    \item $R_d(P_1(x),\cdots ,P_{(d+1)/2}(x))=(P_1(x),\cdots ,P_{(d+1)/2}(x)+x^k).$
\end{enumerate}

Now consider the group $G$ generated by $\{R_i,S_i:1\leq i \leq (d-1)/2\}\cup \{R_d\}$. We agree that the group generated by $R_d$ is commutative, but the other $R_i,S_i's$ generates $k$-step nilpotent groups. Now proceed similarly to the above proof.

\end{proof}

\section*{Acknowledgement} The author is supported by NBHM postdoctoral fellowship with reference no: 0204/27/(27)/2023/R \& D-II/11927. We are thankful to F. K. Richter for his comments on the previous draft of this article.

\end{document}